\documentclass[reqno]{amsart}
\usepackage{amsmath}
\usepackage{amsxtra}
\usepackage{amscd}
\usepackage{amsthm}
\usepackage{amsfonts}
\usepackage{amssymb}
\usepackage{enumitem}
\usepackage{microtype}
\usepackage{graphicx}
\usepackage{epstopdf}
\usepackage{mathrsfs}
\usepackage{eucal} 
\usepackage{comment}
\usepackage[hidelinks, pagebackref, bookmarks=false]{hyperref}
\hypersetup{
    colorlinks,
    citecolor=black,
    filecolor=black,
    linkcolor=blue,
    urlcolor=black,
}
\usepackage{cite}
\usepackage{microtype}
\usepackage{url}

\usepackage{thmtools} 
\usepackage{thm-restate}

\input xy
\xyoption{all}

\numberwithin{equation}{subsection} 
\numberwithin{figure}{subsection} 

\makeatletter
\let\c@equation\c@figure
\makeatother

\newtheorem{theorem}[equation]{Theorem}
\newtheorem{corollary}[equation]{Corollary}

\theoremstyle{remark}
\newtheorem{remark}[equation]{Remark}

\newtheorem{question}[equation]{Question}

\theoremstyle{definition}
\newtheorem{defn}[equation]{Definition}

\newcommand\nc{\newcommand}
\nc\on{\operatorname}

\newcommand\fp{{\mathfrak p}}
\newcommand\fq{{\mathfrak q}}
\newcommand\fm{{\mathfrak m}}

\nc\Hom{\on{Hom}}
\nc\Sections{\on{Sections}}
\nc\Sym{\on{Sym}}
\nc\Spec{\on{Spec}}
\nc\Specm{\on{Specm}}
\nc\ul{\underline}
\nc{\dfp}{\overset{\cdot}{\fp}}
\nc{\dfq}{\overset{\cdot}{\fq}}
\nc{\dfm}{\overset{\cdot}{\fm}}

\begin{document}

\setcounter{tocdepth}{1}

\title{Vanishing for Frobenius Twists of Ample Vector Bundles}
\author{Daniel Litt}

\begin{abstract}
We prove several asymptotic vanishing theorems for Frobenius twists of ample vector bundles in positive characteristic.   As an application, we prove a generalization of the Bott-Danilov-Steenbrink vanishing theorem  for ample vector bundles on toric varieties.
\end{abstract}

\maketitle

\tableofcontents

\section{Introduction}
Let $k$ be a field, and $X$ a finite-type $k$-scheme.  Let $\mathscr{E}$ be a vector bundle on $X$.  In a beautiful series of papers \cite{arapura1, arapura2, arapura3}, Arapura studies a measure of the positivity of $\mathscr{E}$ --- the Frobenius amplitude of $\mathscr{E}$.  Arapura bounds the Frobenius amplitude of ample vector bundles in the case that $k$ has characteristic zero.  The purpose of this paper is to prove bounds on the Frobenius amplitude in \emph{positive} characteristic, under favorable hypotheses, and to give some applications to prove strong vanishing theorems on e.g.~toric varieties.
\begin{defn}[Frobenius Amplitude]
Let $k$ be a field of characteristic $p>0$ and $X$ a finite-type $k$-scheme.  If $\mathscr{E}$ is a vector bundle on $X$, we define $\phi(\mathscr{E})$ to be the least integer $l$ such that for any coherent sheaf $\mathscr{F}$ on $X$, there exists $n\gg 0$ such that $$H^i(X, \mathscr{E}^{(p^n)}\otimes \mathscr{F})=0$$ for all $i>l$.

Now $k$ be a field of characteristic $0$ and $X$ a finite-type $k$-scheme.  For $\mathscr{E}$ a vector bundle on $X$, we say that $\phi(\mathscr{E})<l$ if there exists a finite-type $\mathbb{Z}$-algebra $R$, a finite-type $R$-scheme $\mathcal{X}$, and a vector bundle $\widetilde{\mathscr{E}}$ on $\mathcal{X}$ such that 
\begin{enumerate}
\item there is a ring map $R\to k$ and isomorphisms $X\simeq \mathcal{X}_{k}, \mathscr{E}\simeq \widetilde{\mathscr{E}}_k$, and
\item $\phi(\widetilde{\mathscr{E}}_\mathfrak{q})<l$ for all closed points $\mathfrak{q}\subset \on{Spec}(R).$
\end{enumerate}
\end{defn}

See \cite[Section 1]{arapura1} for more details.  Arapura proves
\begin{restatable*}[Frobenius Amplitude of Ample Bundles, {\cite[Theorem 6.1]{arapura1}}]{theorem}{arapurathmrestate}\label{arapurathm}
Let $k$ be a field of characteristic zero and $X$ a projective variety over $k$.  Let $\mathscr{E}$ be an ample vector bundle on $X$.  Then $$\phi(\mathscr{E})<\on{rk}(\mathscr{E}).$$
\end{restatable*}
We give a simple new proof of this theorem in Section \ref{arapura-proof-section}.

This theorem implies several strong vanishing theorems.  It motivates the following question:
\begin{question}\label{mainquestion}
Let $X$ be a proper variety over a field $k$ of characteristic $p>0$.  Let $\mathscr{E}$ be an ample vector bundle on $X$.  Is $$\phi(\mathscr{E})<\on{rk}(\mathscr{E})?$$
\end{question}
The goal of this paper is to answer Question \ref{mainquestion} under favorable hypotheses --- in particular, we give a new short proof of Theorem \ref{arapurathm}, and answer Question \ref{mainquestion}, if $X$ lifts to $W_2(k)$, and admits a lift of Frobenius.

Explicitly, our main results is:
\begin{restatable*}[Main Theorem]{theorem}{liftablethm}\label{thm:liftable-twist}
Let $k$ be a perfect field of characteristic $p>0$ and $X$ a projective $k$-scheme admitting a flat lift $X_2$ over $W_2(k)$.  Let $\mathscr{E}$ be an ample vector bundle on $X$ with $\on{rk}(\mathscr{E})<p$, and suppose that for some $N>0$, $\mathscr{E}^{(p^N)}$ lifts to $X_2$.  Then $$\phi(\on{Sym}^{n-j}(\mathscr{E})\otimes \bigwedge^j(\mathscr{E}))\leq\on{rk}(\mathscr{E})-j$$ for any $n$.
\end{restatable*}
Taking $n=1, j=1$ we obtain $$\phi(\mathscr{E})< \on{rk}(\mathscr{E}),$$ answering Question \ref{mainquestion} in this case.

We view this as a Manivel-type vanishing theorem.  Unfortunately, the hypotheses are rather strong --- in particular, we must assume that $\mathscr{E}^{(p^N)}$ lifts for some $N>0$ (note that $N=0$ does not suffice).  This is a difficult condition to verify --- the main situation in which it may be checked is when $\mathscr{E}$ admits a lift to $X_2$, and $X_2$ admits a lift of Frobenius:
\begin{restatable*}[Main Corollary]{corollary}{frobcor}\label{thm:frob-lift}
Let $k$ be a perfect field of characteristic $p>0$ and $X$ a projective $k$-scheme admitting a flat lift $X_2$ over $\mathbb{Z}/p^2\mathbb{Z}$; suppose $X_2$ admits a lift of absolute Frobenius.  Let $\mathscr{E}$ be an ample vector bundle on $X$ with $\on{rk}(\mathscr{E})<p$, and suppose that $\mathscr{E}$ lifts to $X_2$.  Then $$\phi(\on{Sym}^{n-j}(\mathscr{E})\otimes \bigwedge^j(\mathscr{E}))\leq\on{rk}(\mathscr{E})-j$$ for any $n$.
\end{restatable*}
The condition that $X_2$ admits a lift of Frobenius is quite restrictive (see e.g. \cite{migliorini}), but it holds for e.g.~toric varieties \cite{buchetal} and canonial lifts of ordinary Abelian varieties.

As an application of these methods, we prove a generalization of the Bott-Steenbrink-Danilov vanishing theorem:
\begin{restatable*}[Generalization of Bott-Steenbrink-Danilov]{theorem}{bottetc}\label{thm:bott-etc}
Let $X$ be a normal projective toric variety over a perfect field $k$, and $\mathscr{E}_1, \cdots, \mathscr{E}_m$ ample vector bundles on $X$.  Let $j: U\hookrightarrow X$ be the inclusion of the smooth locus.  Then if
\begin{enumerate}
\item $\on{char}(k)=0$, or
\item $\on{char}(k)=p>\sum_i\on{rk}(\mathscr{E}_i)$, and each $\mathscr{E}_i$ lifts to the canonical (toric) lift of $X$ to $W_2(k)$,
\end{enumerate}
then $$H^s(X, j_*\Omega^q_U\otimes \on{Sym}^{a_1}\mathscr{E}_1\otimes\cdots \otimes \on{Sym}^{a_m}(\mathscr{E}_m)\otimes \bigwedge^{b_1}\mathscr{E}_1\otimes \cdots \otimes \bigwedge^{b_m} \mathscr{E}_m)=0$$ for $s > \sum_{i=1}^m (\on{rk}(\mathscr{E}_i)-b_i).$
\end{restatable*}
Note that the lifting condition in part (2) is automatically satisfied by toric vector bundles, as they are defined combinatorially.

This result generalizes several results in the literature.  Danilov states (without proof) the special case where $m=1$, $\mathscr{E}_1=\mathscr{L}$ is an ample line bundle and $a_1=0, b_1=1$ \cite[Theorem 7.5.2]{danilov}.  This is proven in the case $X$ is simplicial by Batyrev and Cox \cite{batyrevetal} and in general in \cite{buchetal}.  Manivel \cite{manivel} proves a version of this theorem in characteristic zero, when $X$ is smooth.  Theorem \ref{thm:bott-etc} is a strengthening  (for toric varieties) of his famous vanishing theorem \cite{manivel2}.  See also \cite{brion, mustata, heringetal} for related work.

\subsection{Acknowlegements}  This paper benefited from useful discussions with Donu Arapura, Adam Block, Johan de Jong, Daniel Halpern-Leistner, and John Ottem.
\section{Proof of Theorem \ref{thm:liftable-twist} and Corollary \ref{thm:frob-lift}}
\subsection{The Cartier Isomorphism}
We first recall the main theorem of \cite{deligne-illusie}:
\begin{theorem}[Deligne-Illusie]\label{di-qi}
Let $k$ be a perfect field of characteristic $p>0$ and $S$ a $k$-scheme.  Suppose that $S$ admits a flat lift $\widetilde{S}$ over $W_2(k)$.  Let $X$ be a smooth $S$-scheme with $\dim(X/S)<p$.  Then there is an isomorphism $$F_{X/S*}\Omega^\bullet_{X/S} \simeq \bigoplus \Omega^i_{X'/S}[-i]$$ in $D(X')$ if and only if $X'$ admits a flat lift over $\widetilde{S}$.
\end{theorem}
Here $X'$ is the Frobenius twist of $X$, i.e. it fits into the Cartesian diagram 
$$\xymatrix{
X' \ar[r]\ar[d] & X\ar[d]\\
S \ar[r]^{\on{Frob}} & S
}$$
and $F_{X/S}: X\to X'$ is the relative Frobenius.  

Let $k$ be a perfect field of characteristic $p>0$, and $S$ a $k$-scheme admitting a flat lift $\widetilde{S}$ to $W_2(k)$.  Now suppose $\mathscr{E}$ is a vector bundle on $S$ and set $$\mathbb{E}=\underline{\on{Spec}}_S(\on{Sym}^*(\mathscr{E}))$$ to be the total space of $\mathscr{E}$; let $$\pi: \mathbb{E}\to S$$ be the structure morphism.  Then the Frobenius twist $\mathbb{E}'$ of $\mathbb{E}$ is $$\mathbb{E}'\simeq \underline{\on{Spec}}_S(\on{Sym}^*(\mathscr{E}^{(p)})).$$  Consider the $\mathscr{O}_S$-linear complex $$\pi_*\Omega^\bullet_{\mathbb{E}/S}: 0\to \on{Sym}^*(\mathscr{E})\to \on{Sym}^*(\mathscr{E})\otimes \mathscr{E}\to \on{Sym}^*(\mathscr{E})\otimes \bigwedge^2\mathscr{E}\to \cdots$$
The natural $\mathbb{G}_m$-action on $\mathbb{E}$ makes this into a graded complex; let $$\Omega^\bullet_{\mathbb{E}/S}(n): 0\to \on{Sym}^n(\mathscr{E})\to \on{Sym}^{n-1}(\mathscr{E})\otimes \mathscr{E} \to \cdots\to \on{Sym}^{n-i}(\mathscr{E})\otimes \bigwedge^i\mathscr{E}\to \cdots$$ be the $n$-th graded piece.  By a result of Cartier (see e.g. \cite[Section 7.2]{katz}), there is a natural isomorphism $$\mathscr{H}^i(\Omega^\bullet_{\mathbb{E}/S}(pn))\simeq \on{Sym}^{n-i}(\mathscr{E}^{(p)})\otimes \bigwedge^i \mathscr{E}^{(p)}.$$
By Theorem \ref{di-qi}, this isomorphism may be promoted to an isomorphism in $D^b(S)$ if $\mathscr{E}^{(p)}$ admits a lift to $\widetilde{S}$, and $\on{rk}(\mathscr{E})<p$.  In this case, we have an isomorphism $$\Omega^\bullet_{\mathbb{E}/S}(pn)\simeq \bigoplus_i \on{Sym}^{n-i}(\mathscr{E}^{(p)})\otimes \bigwedge^i \mathscr{E}^{(p)}[-i]$$ in $D^b(S)$. 
\subsection{The Main Theorem}
We are now ready to prove:
\liftablethm
\begin{proof}
Let $\mathscr{F}$ be any coherent sheaf on $X$ and $j\geq 0$ an integer; we wish to show that for $i>\on{rk}(\mathscr{E})-j$, and $m\gg0$, $$H^i(X, \on{Sym}^{n-j}(\mathscr{E}^{(p^m)})\otimes \bigwedge^j \mathscr{E}^{(p^m)}\otimes \mathscr{F})=0.$$
Choose $r\gg 0$ so that $$H^i(X, \on{Sym}^{p^rn-j}(\mathscr{E}^{(p^{N})})\otimes \bigwedge^j \mathscr{E}^{(p^{N})}\otimes \mathscr{F})=0$$ for all $j$ and all $i>0.$  Such an $r$ exists by the ampleness of $\mathscr{E}^{(p^{N})}.$

We claim that 
\begin{equation}\label{mainclaim} 
H^i(X, \on{Sym}^{p^{r-s}n-j}(\mathscr{E}^{(p^{N+s})})\otimes \bigwedge^j \mathscr{E}^{(p^{N+s})}\otimes \mathscr{F})=0 
\end{equation} 
for all $0\leq s\leq r$ and all $i> \on{rk}(\mathscr{E})-j.$  We will prove this by induction on $s$; the case $s=0$ is immediate from our choice of $r$.

Assume that Equation \ref{mainclaim} holds for some $s$; we will prove it for $s+1$.  From the first hypercohomology spectral sequence (associated to the stupid filtration of $\Omega^\bullet_{\mathbb{E}/X}(p^{r-s}n)^{(p^{N+s})}$) $$E_1^{p,q}=H^p(X, \on{Sym}^{p^{r-s}n-q}(\mathscr{E}^{(p^{N+s})})\otimes \bigwedge^q \mathscr{E}^{(p^{N+s})}\otimes \mathscr{F})\implies \mathbb{H}^{p+q}(X, \Omega^\bullet_{\mathbb{E}/X}(p^{r-s}n)^{(p^{N+s})}\otimes \mathscr{F})$$
we have that $$\mathbb{H}^i(X, \Omega^\bullet_{\mathbb{E}/X}(p^{r-s}n)^{(p^{N+s})}\otimes \mathscr{F})=0$$ for $i>\on{rk}(\mathscr{E}),$ by the induction hypothesis.

Now consider the second hypercohomology spectral sequence $$E_2^{p,q}=H^p(X, \mathscr{H}^q(\Omega^\bullet_{\mathbb{E}/X}(p^{r-s}n)^{(p^{N+s})}\otimes \mathscr{F}))\implies \mathbb{H}^{p+q}(X, \Omega^\bullet_{\mathbb{E}/X}(p^{r-s}n)^{(p^{N+s})}\otimes \mathscr{F}).$$  As $\mathscr{E}^{(p^N)}$ is assumed to admit a lift to $W_2$, this spectral sequence is degenerate by Theorem \ref{di-qi}.  Thus $E^{p,q}_2=0$ for $p+q>\on{rk}(\mathscr{E}).$  But by the Cartier isomorphism, 
\begin{align*}
E_2^{p,q}&=H^p(X, \mathscr{H}^q(\Omega^\bullet_{\mathbb{E}/X}(p^{r-s}n)^{(p^{N+s})}\otimes \mathscr{F}))\\
&=H^p(X, \on{Sym}^{p^{r-s-1}n-q}(\mathscr{E}^{(p^{N+s+1})})\otimes \bigwedge^q \mathscr{E}^{(p^{N+s+1})}\otimes \mathscr{F})
\end{align*}
completing the induction step and proving the claim in Equation \ref{mainclaim}.  Taking $r=s$ in \ref{mainclaim} completes the proof.
\end{proof}
The corollary follows easily:
\frobcor
\begin{proof}
By Theorem \ref{thm:liftable-twist}, it suffices to show that $\mathscr{E}^{(p)}$ lifts to $X_2$.  But $\mathscr{E}$ lifts to a vector bundle $\mathscr{E}_2$ on $X_2$, and Frobenius lifts to a $\mathbb{Z}/p^2\mathbb{Z}$-endomorphism of $X_2$, say $F_2$.  Then $F_2^*\mathscr{E}_2$ is a lift of $\mathscr{E}^{(2)}$, as desired.
\end{proof}
\begin{corollary}\label{general-bound}
Let $k, X, X_2$ be as in Corollary \ref{thm:frob-lift}.  Let $\mathscr{E}_1, \cdots, \mathscr{E}_m$ be ample vector vector bundles on $X$ which lift to $X_2$, with $\sum_s\on{rk}(\mathscr{E}_s)<\on{char}(k)$.  Then $$\phi(\on{Sym}^{a_1}\mathscr{E}_1\otimes\cdots \otimes \on{Sym}^{a_m}(\mathscr{E}_m)\otimes \bigwedge^{b_1}\mathscr{E}_1\otimes \cdots \otimes \bigwedge^{b_m} \mathscr{E}_m)\leq \sum_{i=1}^m (\on{rk}(\mathscr{E}_i)-b_i).$$  
\end{corollary}
\begin{proof}
Let $\mathscr{E}=\bigoplus_i \mathscr{E}_i$; let $a=\sum_i a_i, b=\sum_i b_i$.  Then by Corollary \ref{thm:frob-lift}, we have that $$\phi(\Sym^a(\mathscr{E})\otimes \bigwedge^b \mathscr{E})\leq \on{rk}(\mathscr{E})-b=\sum_{i=1}^m (\on{rk}(\mathscr{E}_i)-b_i).$$
But $$\on{Sym}^{a_1}\mathscr{E}_1\otimes\cdots \otimes \on{Sym}^{a_m}(\mathscr{E}_m)\otimes \bigwedge^{b_1}\mathscr{E}_1\otimes \cdots \otimes \bigwedge^{b_m} \mathscr{E}_m$$ is a direct summand of $\Sym^a(\mathscr{E})\otimes \bigwedge^b \mathscr{E}$ so the result follows.
\end{proof}
\section{Proof of Theorem \ref{arapurathm}}\label{arapura-proof-section}
We now give a short proof of Theorem \ref{arapurathm}, originally due to Arapura \cite[Theorem 6.1]{arapura1}.  
\arapurathmrestate
\begin{proof}
By \cite[Lemma 3.3]{arapura1}, there exists $N_0$ such that for all $n>N_0$, $$\phi(\on{Sym}^{n-i}(\mathscr{E})\otimes \bigwedge^i\mathscr{E})=0$$ for all $i$ (using the ampleness of $\mathscr{E}$).  Now let $R$ be a finite-type $\mathbb{Z}$-algebra, with a map $R\to k$ and $(\mathcal{X}, \widetilde{\mathscr{E}})$ a finite-type $R$-scheme with a vector bundle so that $\mathcal{X}_k\simeq X, \widetilde{\mathscr{E}}_k\simeq \mathscr{E}$, and such that $$\phi(\on{Sym}^{n-i}(\widetilde{\mathscr{E}}_{\mathfrak{q}})\otimes \bigwedge^i\widetilde{\mathscr{E}}_{\mathfrak{q}}))=0$$ for all closed points $\mathfrak{q}\in \on{Spec}(R)$ (such a model $(R, \mathcal{X}, \widetilde{\mathscr{E}})$ exists by the definition of $\phi$).  

Let $\mathscr{F}$ be any coherent sheaf on $X$; after replacing $(R, \mathcal{X}, \widetilde{\mathscr{E}})$ by a refinement, we may assume $\mathscr{F}$ extends to a coherent sheaf $\widetilde{\mathscr{F}}$ $\mathcal{X}$.  We claim that for closed points $\mathfrak{q}\in \Spec(R)$ with $p=\on{char}(\kappa(\mathfrak{q}))>N_0$, $j\geq \on{rk}(\mathscr{E})$, and for $m\gg 0$, $$H^j(\mathcal{X}_{\mathfrak{q}}, \widetilde{\mathscr{E}}^{(p^m)}\otimes \widetilde{\mathscr{F}})=0,$$ which clearly suffices.

Consider the hypercohomology spectral sequence $$E_2^{i,j}=H^i(X_{\mathfrak{q}}, \mathscr{H}^j(\Omega^\bullet_{\widetilde{\mathbb{E}}/{X}_{\mathfrak{q}}}(p)^{(p^R)}\otimes \widetilde{\mathscr{F}}))\implies \mathbb{H}^{i+j}(X_{\mathfrak{q}}, \Omega^\bullet_{\widetilde{\mathbb{E}}/{X}_{\mathfrak{q}}}(p)^{(p^R)}\otimes \widetilde{\mathscr{F}}).$$
By the previous paragraph,  the right hand side vanishes for $R\gg 0$, $i+j>\on{rk}(\mathscr{E})$.  On the other hand, $$\mathscr{H}^j(\Omega^\bullet_{\widetilde{\mathbb{E}}/{X}_{\mathfrak{q}}}(p)^{(p^R)}\otimes \widetilde{\mathscr{F}})=\begin{cases} \widetilde{\mathscr{E}}_{\mathfrak{q}}^{(p^{R+1})}\otimes \widetilde{\mathscr{F}} & \text{if $j=0,1$}\\ 0 & \text{otherwise.} \end{cases}$$  Moreover, the only non-zero differentials are on the $E_2$ page for degree reasons, i.e. $$d_2^i: H^i(X_{\mathfrak{q}},\widetilde{\mathscr{E}}_{\mathfrak{q}}^{(p^{R+1})}\otimes \widetilde{\mathscr{F}})\to H^{i+2}(X_{\mathfrak{q}}, \widetilde{\mathscr{E}}_{\mathfrak{q}}^{(p^{R+1})}\otimes \widetilde{\mathscr{F}}).$$  For $R\gg 0$, and $i>\on{rk}(\mathscr{E})-1$, these differentials are thus isomorphisms. Thus for $R\gg 0$, $i>\on{rk}(\mathscr{E})-1$ $$H^i(X_{\mathfrak{q}},\widetilde{\mathscr{E}}_{\mathfrak{q}}^{(p^{R+1})}\otimes \widetilde{\mathscr{F}})=0$$ by backwards induction on $i$, as desired.
\end{proof}
\begin{remark}\label{schur-remark}
In \cite[Proof of Theorem 6.1]{arapura1}, Arapura gives a proof of Theorem \ref{arapurathm} using a resolution of $\mathscr{E}^{(p)}$ by Schur functors, from \cite{carter-lusztig}.  This complex is a subcomplex of $\Omega^\bullet_{\widetilde{\mathbb{E}}/X}(p)$.  Arapura's proof requires the Kempf vanishing theorem as input; we are able to avoid it because of our use of symmetric powers as opposed to other Schur functors.
\end{remark}
\section{Applications}
Let $k$ be a perfect field and $X$ a variety over $k$.  We say that $X$ admits a Frobenius lift if $X$ admits a flat lift $X_2$ to $W_2(k)$, and if absolute Frobenius lifts to $X_2$ as a $\mathbb{Z}/p^2\mathbb{Z}$-morphism.
We first prove the following easy theorem:
\begin{theorem}\label{vanishing}
Let $k$ be a perfect field of characteristic $p>0$, and $X$ a normal projective $k$-variety which admits a Frobenius lift; let $j: U\to X$ be the inclusion of the non-singular locus.  Let $\mathscr{E}$ be a vector bundle on $X$.   Then for all $i$ and all $q> \phi(\mathscr{E})$, $$H^q(X, \mathscr{E}\otimes j_*\Omega^i_{U/k})=0.$$
\end{theorem}
\begin{proof}
We follow the proof of \cite[Theorem 3]{buchetal}.  By \cite[Theorem 2]{buchetal}, there is a split monomorphism of abelian sheaves $$0\to \Omega^i_{U}\to F_*\Omega^i_{U},$$ and hence pushing forward to $X$, a split monomorphism $$0\to j_*\Omega^i_{U}\to F_*j_*\Omega^i_{U}$$ (using that $j$ and $F$ commute).  Tensoring with $\mathscr{E}^{(p^r)}$, we obtain injections $$H^i(X,  j_*\Omega^i_{U}\otimes \mathscr{E}^{(p^r)})\hookrightarrow H^i(X, F_*j_*\Omega^i_{U}\otimes\mathscr{E}^{(p^r)})\simeq H^i(X, j_*\Omega^i_U\otimes \mathscr{E}^{(p^{r+1})})$$ for all $i$, where the latter isomorphism uses the projection formula and the fact that $F$ is affine.  Now we are done by backwards induction on $r$, by the definition of $\phi(\mathscr{E})$.
\end{proof}
We may immediately conclude:
\bottetc
\begin{proof}
By a standard spreading-out argument, it suffices to prove the characteristic $p$ statement.  But this is immediate from Theorem \ref{vanishing} and Corollary \ref{general-bound}.
\end{proof}
\bibliographystyle{alpha}
\bibliography{amplitude}

\begin{thebibliography}{BTLM97}

\bibitem[Ara04]{arapura1}
Donu Arapura.
\newblock Frobenius amplitude and strong vanishing theorems for vector bundles.
\newblock {\em Duke Math. J.}, 121(2):231--267, 2004.
\newblock With an appendix by Dennis S. Keeler.

\bibitem[Ara06]{arapura2}
Donu Arapura.
\newblock Partial regularity and amplitude.
\newblock {\em Amer. J. Math.}, 128(4):1025--1056, 2006.

\bibitem[Ara11]{arapura3}
Donu Arapura.
\newblock Frobenius amplitude, ultraproducts, and vanishing on singular spaces.
\newblock {\em Illinois J. Math.}, 55(4):1367--1384 (2013), 2011.

\bibitem[BC94]{batyrevetal}
Victor~V. Batyrev and David~A. Cox.
\newblock On the {H}odge structure of projective hypersurfaces in toric
  varieties.
\newblock {\em Duke Math. J.}, 75(2):293--338, 1994.

\bibitem[Bri09]{brion}
Michel Brion.
\newblock Vanishing theorems for {D}olbeault cohomology of log homogeneous
  varieties.
\newblock {\em Tohoku Math. J. (2)}, 61(3):365--392, 2009.

\bibitem[BTLM97]{buchetal}
Anders Buch, Jesper~F. Thomsen, Niels Lauritzen, and Vikram Mehta.
\newblock The {F}robenius morphism on a toric variety.
\newblock {\em Tohoku Math. J. (2)}, 49(3):355--366, 1997.

\bibitem[CL74]{carter-lusztig}
Roger~W. Carter and George Lusztig.
\newblock On the modular representations of the general linear and symmetric
  groups.
\newblock {\em Math. Z.}, 136:193--242, 1974.

\bibitem[Dan78]{danilov}
V.~I. Danilov.
\newblock The geometry of toric varieties.
\newblock {\em Uspekhi Mat. Nauk}, 33(2(200)):85--134, 247, 1978.

\bibitem[DI87]{deligne-illusie}
Pierre Deligne and Luc Illusie.
\newblock Rel\`evements modulo {$p^2$} et d\'ecomposition du complexe de de
  {R}ham.
\newblock {\em Invent. Math.}, 89(2):247--270, 1987.

\bibitem[HMP10]{heringetal}
Milena Hering, Mircea Musta\c{t}\u{a}, and Sam Payne.
\newblock Positivity properties of toric vector bundles.
\newblock {\em Ann. Inst. Fourier (Grenoble)}, 60(2):607--640, 2010.

\bibitem[Kat70]{katz}
Nicholas~M. Katz.
\newblock Nilpotent connections and the monodromy theorem: {A}pplications of a
  result of {T}urrittin.
\newblock {\em Inst. Hautes \'Etudes Sci. Publ. Math.}, (39):175--232, 1970.

\bibitem[Man96]{manivel}
Laurent Manivel.
\newblock Th\'eor\`emes d'annulation sur certaines vari\'et\'es projectives.
\newblock {\em Comment. Math. Helv.}, 71(3):402--425, 1996.

\bibitem[Man97]{manivel2}
Laurent Manivel.
\newblock Vanishing theorems for ample vector bundles.
\newblock {\em Invent. Math.}, 127(2):401--416, 1997.

\bibitem[Mig93]{migliorini}
Luca Migliorini.
\newblock Some observations on cohomologically {$p$}-ample bundles.
\newblock {\em Ann. Mat. Pura Appl. (4)}, 164:89--102, 1993.

\bibitem[Mus02]{mustata}
Mircea Musta\c{t}\u{a}.
\newblock Vanishing theorems on toric varieties.
\newblock {\em Tohoku Math. J. (2)}, 54(3):451--470, 2002.

\end{thebibliography}

\end{document}